\documentclass{article}

\usepackage[margin=1.3in]{geometry}

\usepackage{amsmath,amssymb,amsthm,mathrsfs,color,times,textcomp,sectsty,verbatim}
\usepackage{xcolor}
\usepackage[colorlinks=true]{hyperref}
\hypersetup{urlcolor=blue, citecolor=red, linkcolor=blue}
\usepackage[toc,page]{appendix}
\numberwithin{equation}{section}
\theoremstyle{plain}
\newtheorem{theorem}{Theorem}[section]

\newtheorem{lemma}[theorem]{Lemma}
\newtheorem*{conjecture}{Conjecture}

\theoremstyle{definition}

\newtheorem{remark}[theorem]{Remark}

\newcommand\backmatter{\appendix
\def\chaptermark##1{\markboth{%
\ifnum  \c@secnumdepth > \m@ne  \@chapapp\ \thechapter:  \fi  ##1}{%
\ifnum  \c@secnumdepth > \m@ne  \@chapapp\ \thechapter:  \fi  ##1}}%
\def\sectionmark##1{\relax}}

\def \no{\nonumber}
\def \pa{\partial}

\def\e{\varepsilon}

\def\R{\mathbb{R}}
\def\Rn{{\mathbb{R}}^n_+}

\def\D{\Delta}

\def\ba{\begin{align}}
\def\ea{\end{align}}
\def\bp{\begin{proof}}
\def\ep{\end{proof}}

\def\func:u{\bar{u}_{(x_0,\e)}}

\newcommand{\ud}{\mathrm{d}}
 \allowdisplaybreaks

\begin{document}

\title{\Large \bf An existence theorem on the isoperimetric ratio over scalar-flat conformal classes}
\author{Xuezhang  Chen\thanks{X. Chen is partially supported by NSFC (No.11771204), A Foundation for the Author of National Excellent Doctoral Dissertation of China (No.201417) and start-up grant of 2016 Deng Feng program B at Nanjing University. Email: xuezhangchen@nju.edu.cn.}, ~~Tianling Jin\thanks{T. Jin is partially supported by Hong Kong RGC grant GRF 16306918. Email: tianlingjin@ust.hk.} ~~and Yuping Ruan\thanks{Email: ruanyp@umich.edu.}
\medskip
\\
 \small
$^\ast$$^\ddag$Department of Mathematics \& IMS, Nanjing University, Nanjing 210093, P. R. China.\\
 \small$^\dag$Department of Mathematics, Hong Kong University of Science and Technology, \\
  \small Clear Water Bay, Kowloon, Hong Kong. \\
  \small $^\ddag$Department of  Mathematics, University of Michigan, Ann Arbor, MI 48109, USA.\\
}

\date{}

\maketitle

\begin{abstract}
Let $(M,g)$ be a smooth compact Riemannian manifold of dimension $n$ with smooth boundary $\partial M$,  admitting a scalar-flat conformal metric. We prove that the supremum of the isoperimetric ratio  over the scalar-flat conformal class is strictly larger than the best constant of the isoperimetric inequality in the Euclidean space, and consequently is achieved, if either (i) $9\le n\le 11$ and $\partial  M$ has a nonumbilic point; or (ii) $7\le n\le 9$, $\partial M$ is umbilic and the Weyl tensor does not vanish identically on the boundary. This is a continuation of the work \cite{Jin-Xiong} by the second named author and Xiong. 

\medskip


\medskip

\end{abstract}

\section{Introduction}

Let $(M,g)$ be a smooth compact Riemannian manifold of dimension $n\ge 3$ with boundary $\pa M$. In \cite{Hang-Wang-Yan}, F. Hang, X. Wang and X. Yan initiated a study of the isoperimetric quotient  over the scalar-flat conformal class of $g$ on $M$:
\begin{align}\label{theta}
\Theta(M,g)=\sup\left\{\frac{\mathrm{Vol}(M,\tilde g)^\frac{1}{n}}{\mathrm{Vol}(\pa M,\tilde g)^\frac{1}{n-1}}:\  \tilde{g} \in [g]\mathrm{~~with~~} R_{\tilde{g}}=0\right\},
\end{align}
where $[g]=\{\rho^2g: \rho\in C^\infty(M),\rho>0\}$ is the conformal class of $g$, and $R_{g}$ is the scalar curvature of $(M,g)$.  

It was explained in \cite{Hang-Wang-Yan} that the set $\{\tilde{g} \in [g]: R_{\tilde{g}}=0\}$ is not empty if and only if the first eigenvalue $\lambda_1(L_g)$ of the conformal Laplacian
$
L_g:= -\Delta_g +\frac{n-2}{4(n-1)}R_g
$
with zero Dirichlet boundary condition is  positive. Note that the positivity of $\lambda_1(L_g)$ does not depend on the choice of the metrics in $[g]$. Assuming $\lambda_1(L_g)>0$, they proved in \cite{Hang-Wang-Yan} that
  \begin{equation*}
\Theta(\overline {B_1},g_{\R^n})\le \Theta(M,g)<\infty,
\end{equation*}
 and $\Theta(\overline {B_1},g_{\R^n})$ coincides with the best constant of the isoperimetric inequality in the Euclidean space, that is,
 \[
\Theta(\overline {B_1},g_{\R^n})= n^{-\frac{1}{n}}\omega_{n-1}^{-\frac{1}{n(n-1)}},
\]
where $\omega_{n-1}$ is the volume of the unit sphere $\mathbb{S}^{n-1}$. 
They also showed in \cite{Hang-Wang-Yan} that $\Theta(M,g)$ is achieved if the strict inequality
  \begin{equation} \label{HWY:conj}
\Theta(\overline {B_1},g_{\R^n}) < \Theta(M,g)
\end{equation}
  holds, and made a conjecture that:

\begin{conjecture}[\cite{Hang-Wang-Yan}]\label{conjecture} Assume $n\ge 3$, $(M,g)$ is a  smooth compact Riemannian manifold of dimension $n$ with nonempty smooth boundary $\partial M$, and $\lambda_1(L_g)>0$. If $(M,g)$ is not conformally diffeomorphic to $(\overline {B_1}, g_{\R^n})$, then the strict inequality  \eqref{HWY:conj} holds.

\end{conjecture}

In the paper \cite{Jin-Xiong}, the second named author and Xiong verified this conjecture under  one of the following two conditions:
\begin{itemize}
\item $n\ge 12$ and $\partial  M$ has a nonumbilic point;
\item $n\ge 10$, $\partial M$ is umbilic and the Weyl tensor $W_g\neq 0$ at some boundary point.
\end{itemize}
At the same time, Gluck and Zhu \cite{Gluck-Zhu} verified this conjecture  when $M=\overline {B_1}\setminus  B_\varepsilon$ for sufficiently small $\varepsilon>0$ with flat metric in all dimensions.

In this paper, we reduce the dimension assumption in \cite{Jin-Xiong} by three.

\begin{theorem} \label{Thm: J-X refined} 
Let $(M,g)$ be a  smooth compact Riemannian manifold of dimension $n$ with nonempty smooth boundary $\pa M$. Suppose that $\lambda_1(L_g)>0$. If one of the following two conditions 
\begin{itemize}
\item[(i)] $9\le n\le 11 $ and $\partial  M$ has a nonumbilic point;
\item[(ii)] $7\le n\le 9$, $\partial M$ is umbilic, and the Weyl tensor $W_g\neq 0$ at some boundary point;
\end{itemize}
holds, then the strict inequality \eqref{HWY:conj} holds, and consequently, $\Theta(M,g)$ is achieved.
\end{theorem}

 Throughout the paper, we will always assume that $\lambda_1(L_g)>0$. Denote the Poisson kernel of $L_g u=0$ with Dirichlet boundary condition by $P_{g}$. It was pointed out in \cite{Hang-Wang-Yan} that
\begin{equation}\label{eq:energyinf}
\Theta(M,g)=\sup\left\{I[v]: v\in L^{\frac{2(n-1)}{n-2}}(\partial M), v\neq 0\right\},
\end{equation}
where
\[
I[v]=\frac{(\int_{M} |P_{g} v|^{\frac{2n}{n-2}}\,\ud \mu_g)^{\frac{1}{n}}}{(\int_{\partial M} |v|^{\frac{2(n-1)}{n-2}}\,\ud \sigma_g)^{\frac{1}{n-1}}}.
\]
Therefore, to show the strict inequality \eqref{HWY:conj}, we need to find a test function $v\in L^{\frac{2(n-1)}{n-2}}(\partial M)$ such that
\begin{equation}\label{eq:testinequality}
I[v]>\Theta(\overline {B_1},g_{\R^n}).
\end{equation}

Recall that it was shown in Hang-Wang-Yan \cite[Theorem 1.1]{Hang-Wang-Yan2} that 
$$
\Theta(\overline {B_1},g_{\R^n})=\Theta(\overline{\R^n_+},g_{\R^n}),
$$ 
where $\R^n_+=\{x=(x',x_n)\in\R^n: x'\in\R^{n-1},x_n>0\}$ is the upper half-space, and moreover, $\Theta(\overline{\R^n_+},g_{\R^n})$ defined as in \eqref{eq:energyinf} is achieved by  the so-called bubbles:
\begin{equation}\label{eq:bubble}
c\left(\frac{\e}{\e^2+|x'-\xi_0|^2}\right)^{\frac{n-2}{2}},
\end{equation}
where $c\in\R_+$, $\e>0$ and $\xi_0\in\R^{n-1}$.  The test function $v$ chosen in  \cite{Jin-Xiong} to verify \eqref{eq:testinequality} is a cut-off of the bubbles \eqref{eq:bubble} in proper coordinates on $M$ centered at a boundary point.  In our proof of Theorem \ref{Thm: J-X refined}, we choose the same test function, but we will give a more delicate calculation of the $L^{\frac{2n}{n-2}}$-norm of its Poisson extension $P_gv$.

One sees from the definition that $\Theta(M,g)$ depends only on the conformal class $[g]$. These results on the above variational problem \eqref{eq:energyinf}  show an analogy to the Yamabe problem solved by Yamabe \cite{Yamabe}, Trudinger \cite{Trudger}, Aubin \cite{Aubin} and Schoen \cite{Schoen}, as well as to the boundary Yamabe problem (or higher dimensional Riemannian mapping problem) studied by Escobar \cite{escobar1,escobar6}, Marques \cite{marques1, marques3}, Han-Li \cite{han-li1,han-li2}, Chen \cite{ChenSophie}, Almaraz \cite{Almaraz1}, Mayer-Ndiaye \cite{MN}, Chen-Ruan-Sun \cite{Chen-Ruan-Sun}, etc. A prescribing function problem of the isoperimetric ratio on the unit sphere, which is a Nirenberg type problem, has been studied by Xiong \cite{Xiong}. 

This paper is organized as follows. In the next section, we will review the proof in \cite{Jin-Xiong}.  In Section \ref{sec:proof}, we will first set up our objectives on how to reduce the dimension assumption in \cite{Jin-Xiong}, and carry out our detailed calculations afterwards.

 \bigskip

\noindent \textbf{Acknowledgement:} Part of this work was completed while the first named author was visiting the Department of Mathematics at the Hong Kong University of Science and Technology, to which  he is grateful   for providing  the very stimulating research environment and supports. We all would like to thank Professor YanYan Li for his interests and constant encouragement.

\section{An overview}\label{sec:review}

The purpose of this section is to summarize the proof of \cite[Theorem 1.2]{Jin-Xiong}, on which our calculations are based.

Let $P\in\pa M$ be a non-umbilic point in case (i) of Theorem \ref{Thm: J-X refined}, or a point at which the Weyl tensor of $M$ does not vanish in case (ii). Since $\Theta(M,g)$ is a conformal invariant, we can choose conformal Fermi coordinates (see Marques \cite{marques1}) $x=(x',x_n)$ centered at $P$ to simplify the computations.

For any fixed $1\gg \rho\gg \e>0$, denote by $\chi_\rho(t)$ a smooth cut-off function supported on $[0,2\rho]$ such that $\chi_\rho(t)=1$ in $[0,\rho]$ and $0<\chi_\rho(t)<1$ in $(\rho,2\rho)$. Let
$$v_\e(x')=\left(\frac{\e}{\e^2+|x'|^2}\right)^\frac{n-2}{2}\chi_\rho(|x'|)$$
in the above coordinates. Let $u_\e$ be the $L_g$-harmonic extension of $v_\e$ in $M$, or equivalently, the solution to 
$$L_gu_\e=0\quad\mathrm{in~~}M,\quad u_\e=v_\e\quad\mathrm{on~~}\pa M.$$
Roughly, we can regard $u_\e$ in $(M,g)$ as a small perturbation of the harmonic extension of $\left(\frac{\e}{\e^2+|x'|^2}\right)^\frac{n-2}{2}$ in the Euclidean upper half-space $\R_+^n$, which is
\[
\overline U_\e(x)=\left(\frac{\e}{(x_n+\e)^2+|x'|^2}\right)^\frac{n-2}{2}.
\]
Thus, it is natural to consider the error term $W_\e:=u_\e-\overline{U}_\e$, which satisfies
\begin{align*}
\begin{cases}
\displaystyle L_gW_\e=-L_g\overline{U}_\e=:F[\overline{U}_\e] \quad&\mathrm{in~~} \Omega,\\
\displaystyle W_\e(x',0)=0 &\mathrm{on~~}\pa\Omega\cap\pa M, \\
\displaystyle W_\e=u_\e-\overline{U}_\e &\mathrm{on~~}\pa\Omega\setminus\pa M.
\end{cases}
\end{align*}
Here $\Omega$ is a smooth domain in $\Rn$ such that $B_\rho^+(0)\subset \Omega\subset B_{2\rho}^+(0)$ and radially symmetric with respect to $x'$. 
Denote by $\mathcal{G}(f)$ the solution to 
$$-\D u=f\quad\mathrm{in~~}\Omega,\quad u=0\quad\mathrm{on~~}\pa\Omega.$$
Decompose $W_\e$ into three parts:
\begin{align}\label{JX-error decomposition}
W_\e^{(1)}=\mathcal{G}(F[\overline{U}_\e]),\quad W_\e^{(2)}=\mathcal{G}(F[W_\e^{(1)}]) \quad\mathrm{and}\quad W_\e^{(3)}=W_\e-W_\e^{(1)}-W_\e^{(2)}.
\end{align}   
It has been obtained in \cite{Jin-Xiong} that
\begin{align}\label{JX-volume expansion}
&\int_{M}|u_\e|^\frac{2n}{n-2}\ud \mu_g \nonumber \\
=&\frac{1}{n2^n}\omega_{n-1}+\frac{2n}{n-2}\int_\Omega\overline{U}_\e^\frac{n+2}{n-2}W_\e^{(1)}\ud x+\frac{2n}{n-2}\int_\Omega\overline{U}_\e^\frac{n+2}{n-2}W_\e^{(2)}\ud x++\frac{2n}{n-2}\int_\Omega\overline{U}_\e^\frac{n+2}{n-2}W_\e^{(3)}\ud x\no\\
&+\frac{n(n+2)}{(n-2)^2}\int_\Omega\overline{U}_\e^\frac{4}{n-2}W_\e^2\ud x+h.o.t.,
\end{align}
and
\begin{align}\label{JX-area expansion}
\int_{\pa M}v_\e^\frac{2(n-1)}{n-2}\ud\sigma_g=2^{1-n}\omega_{n-1}+O(\e^{n-1}).
\end{align}

\textbf{Case 1.} If $\pa M$ admits a non-umbilic point $P$, then it follows from \cite[(19), (22), (24)]{Jin-Xiong} that for $n\ge 5$ and $s=1,2,3,4$,
\begin{equation}\label{eq:estimateofleading}
\begin{split}
|W_\e^{(1)}|+|x+\e e_n|^{s}|\nabla^s W_\e^{(1)}|& \le  C \e^{\frac{n-2}{2}} |x+\e e_n|^{3-n},\\
|W_\e^{(2)}|+|x+\e e_n|^{s}|\nabla^s W_\e^{(2)}|& \le  C \e^{\frac{n-2}{2}} |x+\e e_n|^{4-n},\\
|W_\e^{(3)}|& \le  
\begin{cases}
C \e^{\frac{n-2}{2}} |x+\e e_n|^{5-n}&\quad\mathrm{if~~}n\ge 6,\\
C \e^{\frac{n-2}{2}} |\log\e|&\quad\mathrm{if~~}n=5,
\end{cases}
\end{split}
\end{equation}
and from \cite[(30), (37), (38)]{Jin-Xiong} that
\begin{align*}
\frac{2n}{n-2}\int_\Omega\overline{U}_\e^\frac{n+2}{n-2}W_\e^{(1)}\ud x 
=&\frac{\omega_{n-2}(n-12)B(\frac{n-1}{2},\frac{n+1}{2})}{4n(n-1)(n-2)(n-3)}|h|^2\e^2+O(\e^3), \\
\frac{2n}{n-2}\int_\Omega\overline{U}_\e^\frac{n+2}{n-2}W_\e^{(2)}\ud x 
=&\frac{8n^2(n+2)}{3}|h|^2\e^2\int_{\Rn}x^2_n|x+e_n|^{-(n+4)}V(|x'|,x_n)x_1^4 \ud x +O(\e^3),\\
\frac{2n}{n-2}\int_\Omega\overline{U}_\e^\frac{n+2}{n-2}W_\e^{(3)}\ud x 
=&O(\e^3),\\
\frac{n(n+2)}{(n-2)^2}\int_\Omega\overline{U}_\e^\frac{4}{n-2}W_\e^2\ud x 
=&\frac{8n^3(n+2)}{3}|h|^2\e^2\int_{\Rn}|x+e_n|^{-4}V^2(|x'|,x_n)x_1^4\ud x+O(\e^3),
\end{align*}
where $h$ is the second fundamental form at $P$ with respect to the outward unit normal vector and $V$ is a positive function. Moreover, if we define 
\[
\widetilde{V}(x)=V(|x'|,x_{n+4})
\] 
($x\in\R^{n+4}_+$, $x'=(x_1,\cdots,x_{n+3})$) as a function in $\R^{n+4}_+$, then $\widetilde{V}$ satisfies (see \cite[(35)]{Jin-Xiong})
\begin{align}\label{eq:subsolution1}
\begin{cases}
\displaystyle -\D \widetilde V=x_{n+4}|x+e_{n+4}|^{-n-2}, &\quad \mathrm{in~~} \R^{n+4}_+, \\
\displaystyle  \widetilde V=0,   &\quad \mathrm{on~~} \pa \R^{n+4}_+.  
\end{cases}
\end{align}
Therefore, from \eqref{JX-volume expansion} we have
\[
\int_{M}|u_\e|^\frac{2n}{n-2}\ud \mu_g =\frac{1}{n2^n}\omega_{n-1} + C_1(n)\e^2|h|^2+h.o.t.,
\]
where
\begin{equation}\label{eq:C1}
\begin{split}
C_1(n):=&\frac{\omega_{n-2}(n-12)B(\frac{n-1}{2},\frac{n+1}{2})}{4n(n-1)(n-2)(n-3)}\\
&+\frac{8n^2(n+2)}{3}\int_{\Rn}x^2_n|x+e_n|^{-(n+4)}V(|x'|,x_n)x_1^4 \ud x \\
&+\frac{8n^3(n+2)}{3}\int_{\Rn}|x+e_n|^{-4}V^2(|x'|,x_n)x_1^4\ud x,
\end{split}
\end{equation}
and $B(\cdot,\cdot)$ is the Beta function. Since $|h|^2>0$ in this case, there holds $C_1(n)>0$ when $n \geq 12$, from which the result in \cite{Jin-Xiong} follows.

\medskip

\textbf{Case 2.} Assume that $\pa M$ is umbilic and the Weyl tensor $W_g$ of $M$ is nonzero at some boundary point $P$, then it follows from \cite[(43)]{Jin-Xiong} that for $n\ge 7$ and $s=1,2,3,4$,
\begin{equation}\label{eq:estimateofleading2}
\begin{split}
|W_\e^{(1)}|+|x+\e e_n|^{s}|\nabla^s W_\e^{(1)}|& \le  C \e^{\frac{n-2}{2}} |x+\e e_n|^{4-n},\\
|W_\e^{(2)}|+|x+\e e_n|^{s}|\nabla^s W_\e^{(2)}|& \le  C \e^{\frac{n-2}{2}} |x+\e e_n|^{6-n},\\
|W_\e^{(3)}|& \le   
\begin{cases}
C\e^{\frac{n-2}{2}} |x+\e e_n|^{8-n}&\quad\mathrm{if~~}n\ge 9,\\
C\e^{\frac{n-2}{2}} |\log\e|&\quad\mathrm{if~~}n= 8,\\
C\e^{\frac{n-2}{2}} &\quad\mathrm{if~~}n= 7,
\end{cases}
\end{split}
\end{equation}
and from\cite[(46), (51), (52)]{Jin-Xiong} that
\begin{align*}
\frac{2n}{n-2}\int_\Omega\overline{U}_\e^\frac{n+2}{n-2}W_\e^{(1)}\ud x 
=&\frac{3(n-10)\omega_{n-2}B(\frac{n-1}{2},\frac{n+1}{2})}{2n(n-1)(n-2)(n-3)(n-4)(n-5)}(R_{ninj})^2\e^4+a(n)|\overline{W}|^2\e^4+O(\e^5),\\
\frac{2n}{n-2}\int_\Omega\overline{U}_\e^\frac{n+2}{n-2}W_\e^{(2)}\ud x  
=&\frac{2n^2(n+2)}{3}(R_{ninj})^2\e^4\int_{\R^n_+}|x+e_n|^{-n-4}\Lambda(|x'|,x_n)x_n^3x_1^4\ud x +O(\e^5),\\
\frac{2n}{n-2}\int_\Omega\overline{U}_\e^\frac{n+2}{n-2}W_\e^{(3)}\ud x  
=&O(\e^5),\\
\frac{n(n+2)}{(n-2)^2}\int_\Omega\overline{U}_\e^\frac{4}{n-2}W_\e^2\ud x 
&=\frac{2n^3(n+2)}{3}(R_{ninj})^2\e^4\int_{\R^n_+}|x+e_n|^{-4}x_1^4\Lambda^2(|x'|,x_n)\ud x+O(\e^5),
\end{align*}
where $a(n)>0$ is a constant, $\overline{W}$ is the Weyl tensor of $\pa M$ with the induced metric of $g$,  and $\Lambda$ is a positive function. Moreover, if we define 
\[
\widetilde{\Lambda}(x)=\Lambda(|x'|,x_{n+4})
\] 
as a function in $\R^{n+4}_+$, then $\tilde{\Lambda}$ satisfies  (see \cite[(50)]{Jin-Xiong})
\begin{align}\label{eq:subsolution2}
\begin{cases}
\displaystyle -\D \widetilde \Lambda=x_{n+4}^2|x+e_{n+4}|^{-n-2}, &\quad \mathrm{in~~} \R^{n+4}_+, \\
\displaystyle  \widetilde \Lambda=0,   &\quad \mathrm{on~~} \pa \R^{n+4}_+.  
\end{cases}
\end{align}
Therefore, from \eqref{JX-volume expansion} we have
\[
\int_{M}|u_\e|^\frac{2n}{n-2}\ud \mu_g =\frac{1}{n2^n}\omega_{n-1} + C_2(n)\e^4(R_{ninj})^2+a(n)|\overline{W}|^2\e^4+h.o.t.,
\]
where
\begin{equation}\label{eq:C2}
\begin{split}
C_2(n):=&\frac{3(n-10)\omega_{n-2}B(\frac{n-1}{2},\frac{n+1}{2})}{2n(n-1)(n-2)(n-3)(n-4)(n-5)}\\
&+\frac{2n^2(n+2)}{3}\int_{\R^n_+}|x+e_n|^{-n-4}\Lambda(|x'|,x_n)x_n^3x_1^4\ud x \\
&+\frac{2n^3(n+2)}{3}\int_{\R^n_+}|x+e_n|^{-4}x_1^4\Lambda^2(|x'|,x_n)\ud x,
\end{split}
\end{equation}
and $B(\cdot,\cdot)$ is the Beta function. It is known from  Almaraz \cite[Lemma 2.5]{Almaraz1} that under conformal Fermi coordinates around $P$, $|\overline W|^2+(R_{ninj})^2\neq 0$ at $P$ is equivalent to the Weyl tensor $|W_g|\neq 0$ at $P$. 
Therefore, if $n\ge 10$, then we have $C_2(n)>0$, from which the result in \cite{Jin-Xiong} follows.

 \section{Proofs}\label{sec:proof}
As mentioned in \cite[Remarks 3.4 and 4.2]{Jin-Xiong}, one may reduce the dimension assumptions in \cite{Jin-Xiong} if one can explicitly calculate, or obtain useful lower bounds of $\widetilde V$ and $\widetilde \Lambda$, that are solutions of \eqref{eq:subsolution1} and \eqref{eq:subsolution2} respectively.

In this paper, we find a way of obtaining some useful subsolutions to \eqref{eq:subsolution1} and \eqref{eq:subsolution2}, which serve as lower bounds of $\widetilde V$ and $\widetilde \Lambda$, respectively.  These give better estimates of the constants $C_1(n)$ and $C_2(n)$, defined in \eqref{eq:C1} and \eqref{eq:C2}, respectively, which in return reduce the dimension assumptions in \cite{Jin-Xiong}.

\subsection{Two calculus lemmas}

The Laplacian operator in $\R^{n+4}$ applying to functions that are radial in the first $n+3$ variables (denoting $r=|(x_1,\cdots, x_{n+3})|, s=x_{n+4}$) is
\[
\Delta=\frac{\partial^2 }{\partial r^2}+\frac{n+2}{r} \frac{\partial}{\partial r}+\frac{\partial^2 }{\partial s^2}.
\]
All of our calculations are based on the next calculus lemma.

\begin{lemma}\label{lem:calculus}
For $\phi(r,s)=f(s)(r^2+(1+s)^2)^{-\frac{\alpha}{2}}$, $\alpha\in\mathbb{R}, r\ge 0,s\ge 0$, we have
\begin{align*}
&-\left(\phi_{rr}+\frac{n+2}{r}\phi_r+\phi_{ss}\right)\\
=&\Big[\alpha (n+2-\alpha) f(s) +2\alpha(1+s)f'(s)-f''(s) (r^2+(1+s)^2)\Big](r^2+(1+s)^2)^{-\frac{\alpha+2}{2}}.
\end{align*}
In particular, if $f$ is convex, then we have
\begin{align*}
-\left(\phi_{rr}+\frac{n+2}{r}\phi_r+\phi_{ss}\right)\le \Big[\alpha (n+2-\alpha) f(s) +2\alpha(1+s)f'(s)\Big](r^2+(1+s)^2)^{-\frac{\alpha+2}{2}}.
\end{align*}
\end{lemma}
The proof of this lemma is elementary, and we omit the details.

We also summarize the following calculation as a lemma, which will be frequently used as well. It is essentially just the integration by parts.
\begin{lemma}\label{lem:changeofvariables}
Suppose $\phi_1,\phi_2\in C^2(\overline \R_+\times\overline \R_+)$ are two nonnegative functions with sufficiently fast decay at infinity. For $i=1,2$, let $\tilde u_i$ be the unique solution (decay to zero at infinity) of
\begin{align*}
\begin{cases}
\displaystyle -\D \tilde u_i=\phi_i(|x'|,x_{n+4}),& \mathrm{in~~} \R^{n+4}_+, \\
\displaystyle \tilde u_i=0, & \mathrm{on~~} \pa \R^{n+4}_+,
\end{cases}
\end{align*}
where $x'=(x_1,\cdots,x_{n+3})$ for $x\in \R^{n+4}$. For $r,s\ge 0$, let $u_i(r,s)=\tilde u_i(r,0,\cdots,0, s)$. Then
\[
\int_{\Rn}\phi_1(|x'|,x_n)u_2(|x'|,x_n)x_1^4 \ud x= \int_{\Rn}\phi_2(|x'|,x_n)u_1(|x'|,x_n)x_1^4 \ud x.
\]
\end{lemma}
\begin{proof}
From the Green's function of $-\Delta$ in $\R^{n+4}_+$ with zero Dirichlet boundary condition, we know that $\tilde u_i$ is radial in $x'$. Using this symmetry and integration by parts, we have
\begin{align*}
&\int_{\Rn}\phi_1(|x'|,x_n)u_2(|x'|,x_n)x_1^4 \ud x\\
=&\frac{3\omega_{n-2}}{(n-1)(n+1)\omega_{n+2}}\int_{\mathbb{R}^{n+4}_+}\phi_1(|x'|,x_{n+4})u_2(|x'|,x_{n+4})\ud x\\
=&\frac{3\omega_{n-2}}{(n-1)(n+1)\omega_{n+2}}\int_{\mathbb{R}^{n+4}_+}(-\Delta\tilde u_1) \tilde u_2 \ud x\\
=&\frac{3\omega_{n-2}}{(n-1)(n+1)\omega_{n+2}}\int_{\mathbb{R}^{n+4}_+}(-\Delta\tilde u_2) \tilde u_1 \ud x\\
=&\int_{\Rn}\phi_2(|x'|,x_n)u_1(|x'|,x_n)x_1^4 \ud x.
\end{align*}
This finishes the proof.
\end{proof}
For convenience, we did not explicitly assume the required decay rates on $\phi_1$ and $\phi_2$ in  Lemma \ref{lem:changeofvariables}. However, when this lemma is applied in the next two sections, it will be clear that the decay rates there will be sufficient.

\subsection{Non-umbilic boundary in dimensions $9 \leq n \leq 11$}
As stated earlier, to estimate $\widetilde V$ defined in \eqref{eq:subsolution1}, we want to find a sub-solution of
\begin{align*}
\begin{cases}
\displaystyle -\D \widetilde V=x_{n+4}|x+e_{n+4}|^{-n-2}, &\quad \mathrm{in~~} \R^{n+4}_+, \\
\displaystyle  \widetilde V=0,   &\quad \mathrm{on~~} \pa \R^{n+4}_+.  
\end{cases}
\end{align*}
According to Lemma \ref{lem:calculus}, we will choose $\alpha=n$, and search for the solution of
\begin{equation}\label{eq:ode1}
\begin{cases}
\alpha (n+2-\alpha) f(s) +2\alpha(1+s)f'(s)=s,\\
f(0)=0.
\end{cases}
\end{equation}
The solution of \eqref{eq:ode1} is 
\[
\frac{s^2}{4n(1+s)},
\]
which is a convex function. So we have
\begin{align*}
-\D \left(\frac{x_{n+4}^2}{4n(1+x_{n+4})}|x+e_{n+4}|^{-n}\right)=x_{n+4}|x+e_{n+4}|^{-n-2}-\frac{1}{2n(1+x_{n+4})^3}|x+e_{n+4}|^{-n}.
\end{align*}
Set
$$\tilde u_1(x)=\widetilde V(x)-\frac{1}{4n}\frac{x_{n+4}^2}{1+x_{n+4}}|x+e_{n+4}|^{-n},\quad x\in \R^{n+4}_+.$$
Then
\begin{align}\label{eq:u_1}
\begin{cases}
\displaystyle-\D \tilde u_1
   = \frac{1}{2n(1+x_{n+4})^3}|x+e_{n+4}|^{-n}, &\mathrm{in~~} \mathbb{R}_+^{n+4},\\
\tilde u_1=0, &\mathrm{on~~} \pa \mathbb{R}_+^{n+4}.
\end{cases}
\end{align}
Thus, it follows from the maximum principle that
\begin{align}\label{ineq:subsol:1}
\tilde u_1(x)= \widetilde V(x) - \frac{1}{4n}\frac{x_{n+4}^2}{1+x_{n+4}}|x+e_{n+4}|^{-n}\geq 0 \quad \mathrm{in~~} \overline{\mathbb{R}_+^{n+4}}. \end{align}

Now we are going to give a better estimate of  $C_1(n)$ defined in \eqref{eq:C1}, which is
\begin{equation*}
\begin{split}
C_1(n)=&\frac{\omega_{n-2}(n-12)B(\frac{n-1}{2},\frac{n+1}{2})}{4n(n-1)(n-2)(n-3)}\\
&+\frac{8n^2(n+2)}{3}\int_{\Rn}x^2_n|x+e_n|^{-(n+4)}V(|x'|,x_n)x_1^4 \ud x \\
&+\frac{8n^3(n+2)}{3}\int_{\Rn}|x+e_n|^{-4}V^2(|x'|,x_n)x_1^4\ud x.
\end{split}
\end{equation*}

Let
\[
u_1(r,s)= \tilde u_1 (r,0,\cdots,0,s),
\]
and recall
\[
V(r,s)= \widetilde V(r,0,\cdots,0,s).
\]
By \eqref{ineq:subsol:1} we have
\begin{equation}\label{eq:linearV}
\begin{split}
&\int_{\Rn}x^2_n|x+e_n|^{-(n+4)}V(|x'|,x_n)x_1^4 \ud x \\
=&\frac{1}{4n}\int_{\Rn}x^2_n|x+e_n|^{-n-4}\frac{x_n^2}{1+x_n}|x+e_n|^{-n}x_1^4 \ud x +\int_{\Rn}x^2_n|x+e_n|^{-(n+4)}u_1(|x'|,x_n)x_1^4 \ud x.
\end{split}
\end{equation}
For the first term, we have
\begin{align}
\frac{1}{4n}\int_{\Rn}|x+e_n|^{-2n-4}\frac{x_1^4x_n^4}{1+x_n} \ud x 
=&\frac{3\omega_{n-2}}{4n(n-1)(n+1)}\int_0^{\infty}\frac{x_n^4\ud x_n}{(1+x_n)^{n+2}}\int_0^{\infty}\frac{r^{n+2}\ud r}{(1+r^2)^{n+2}}\nonumber \\
=&\frac{9\omega_{n-2}B(\frac{n-1}{2},\frac{n+1}{2})}{4(n+1)^2n^3(n-1)(n-2)(n-3)}.\label{eq:linearV1}
\end{align}
For the second term, if we let
$\tilde v_1$ be the solution (decay to zero at infinity) of
\begin{align*}
\begin{cases}
\displaystyle -\D \tilde v_1=x_{n+4}^2|x+e_{n+4}|^{-n-4},& \mathrm{in~~} \R^{n+4}_+, \\
\displaystyle \tilde v_1=0, & \mathrm{on~~} \pa \R^{n+4}_+,
\end{cases}
\end{align*}
and notice that $\tilde v_1$ is radial in the $x'=(x_1,\cdots,x_{n+3})$ variable, then it follows from Lemma \ref{lem:changeofvariables} and \eqref{eq:u_1} that

\begin{align*}
\int_{\Rn}x^2_n|x+e_n|^{-(n+4)}u_1(|x'|,x_n)x_1^4 \ud x
&=\frac{1}{2n}\int_{\Rn}v_1(|x'|,x_n)\frac{1}{(1+x_{n})^3}|x+e_{n}|^{-n}x_1^4\ud x,
\end{align*}
where 
\[
v_1(r,s)= \tilde v_1 (r,0,\cdots,0,s).
\]

Next we give some lower bound estimates of $\tilde v_1$.  According to Lemma \ref{lem:calculus}, we will choose $\alpha=n+2$,  and search for the solution of
\begin{equation}\label{eq:ode2}
\begin{cases}
2\alpha(1+s)f'(s)=s^2,\\
f(0)=0.
\end{cases}
\end{equation}
The solution of \eqref{eq:ode2} is 
\[
\frac{1}{2(n+2)}\left[\frac 12 s^2-s+\log(1+s)\right],
\]
which is a convex function. Hence, we have in $\R^{n+4}_+$ that
\begin{align*}
&- \D\left\{\frac{1}{2(n+2)}\left[\frac{1}{2}x_{n+4}^2-x_{n+4}+\log(1+x_{n+4})\right]|x+e_{n+4}|^{-n-2}\right\} \le x_{n+4}^2|x+e_{n+4}|^{-n-4}.
\end{align*}
So it follows from the maximum principle that
\begin{align}\label{ineq:subsol:2}
\tilde v_1(x)\geq \frac{1}{2(n+2)}\left[\frac{1}{2}x_{n+4}^2-x_{n+4}+\log(1+x_{n+4})\right]|x+e_{n+4}|^{-n-2} \geq 0 \quad \mathrm{in~~} \overline{\R^{n+4}_+}.
\end{align}
Therefore,
\begin{align}
&\int_{\Rn}x^2_n|x+e_n|^{-(n+4)}u_1(|x'|,x_n)x_1^4 \ud x\nonumber\\
\ge &\frac{1}{4n(n+2)}\int_{\Rn}\left[\frac{1}{2}x_{n}^2-x_{n}+\log(1+x_{n})\right]|x+e_{n}|^{-2n-2}\frac{x_1^4}{(1+x_{n})^3}\ud x\nonumber \\
=&\frac{3\omega_{n-2}}{4n(n+2)(n-1)(n+1)}
\int_0^{\infty}\frac{\frac{1}{2}x_{n}^2-x_{n}+\log(1+x_{n})}{(1+x_n)^{n+2}}\ud x_n\int_0^\infty\frac{r^{n+2}\ud r}{(1+r^2)^{n+1}} \nonumber\\
=&\frac{3\omega_{n-2}}{4(n+2)(n+1)n(n-1)}\frac{2}{(n+1)^2n(n-1)}\frac{n+1}{4n}B\left(\tfrac{n-1}{2},\tfrac{n+1}{2}\right) \nonumber\\
=&\frac{3\omega_{n-2}B\left(\frac{n-1}{2},\frac{n+1}{2}\right)}{8(n+2)(n+1)^2n^3(n-1)^2}.\label{eq:linearV2}
\end{align}

By \eqref{ineq:subsol:1} we obtain
\begin{align}
&\int_{\Rn}|x+e_n|^{-4}V^2(|x'|,x_n)x_1^4\ud x\nonumber\\
=&\int_{\Rn}|x+e_n|^{-4}\left(\frac{1}{4n}\frac{x_n^2}{1+x_n}|x+e_n|^{-n}\right)^2x_1^4 \ud x +\frac{1}{2n}\int_{\Rn}|x+e_n|^{-n-4}\frac{x_n^2}{1+x_n}x_1^4 u_1(|x'|,x_n) \ud x \nonumber\\
&+\int_{\Rn}|x+e_n|^{-4}u_1^2(|x'|,x_n) x_1^4 \ud x\nonumber\\
=& I_1+I_2+I_3.\label{eq:quarV}
\end{align}
For $I_1$, we have
\begin{align}
I_1=&\frac{3\omega_{n-2}}{16n^2(n+1)(n-1)}\int_0^{\infty}\frac{x_n^4 \ud x_n}{(1+x_n)^{n+3}}\int_0^{\infty}\frac{r^{n+2}\ud r}{(1+r^2)^{n+2}} \nonumber\\
=&\frac{9\omega_{n-2}B(\frac{n-1}{2},\frac{n+1}{2})}{16(n+2)(n+1)^2n^4(n-1)(n-2)}.\label{eq:quarV1}
\end{align}

For $I_2$, if we let $ \tilde w_1$ be the solution (decay to zero at infinity) of
\begin{align*}
\begin{cases}
\displaystyle -\D \tilde w_1=\frac{x_{n+4}^2}{1+x_{n+4}}|x+e_{n+4}|^{-n-4},& \mathrm{in~~} \R^{n+4}_+, \\
\displaystyle \tilde w_1=0, & \mathrm{on~~} \pa \R^{n+4}_+,
\end{cases}
\end{align*}
and notice that $\tilde w_1$ is radial in the $x'=(x_1,\cdots,x_{n+3})$ variable, then it follows from Lemma \ref{lem:changeofvariables} and \eqref{eq:u_1} that
 \begin{align*}
I_2=\frac{1}{4n^2}\int_{\mathbb{R}^n_+}w_1(|x'|,x_n)(1+x_n)^{-3}|x+e_n|^{-n}x_1^4\ud x,
\end{align*}
where 
\[
w_1(r,s)= \tilde w_1 (r,0,\cdots,0,s).
\]

We will give a lower bound estimate for $\tilde w_1$.  According to Lemma \ref{lem:calculus}, we will choose $\alpha=n+2$,  and search for the solution of
\begin{equation}\label{eq:ode3}
\begin{cases}
2\alpha(1+s)f'(s)=\frac{s^2}{1+s},\\
f(0)=0.
\end{cases}
\end{equation}
The solution of \eqref{eq:ode3} is 
\[
\frac{1}{2(n+2)}\left[1+ s- 2\log(1+s)-\frac{1}{1+s}\right],
\]
which is a convex function. Hence, we have
\begin{align*}
&- \D\left\{\frac{1}{2(n+2)}\left[1+x_{n+4}-2\log(1+x_{n+4})-\frac{1}{1+x_{n+4}}\right]|x+e_{n+4}|^{-n-2}\right\} \\
&\le \frac{x_{n+4}^2}{1+x_{n+4}}|x+e_{n+4}|^{-n-4}\quad \mathrm{in~~} \R^{n+4}_+.
\end{align*}
So it follows from the maximum principle that
\[
\tilde w_1(x)\ge \frac{1}{2(n+2)}\left(1+x_{n+4}-2\log(1+x_{n+4})-\frac{1}{1+x_{n+4}}\right)|x+e_{n+4}|^{-n-2}\ge 0 \quad \mathrm{in~~} \overline{\R^{n+4}_+}.
\]
Thus,
\begin{align}
I_2\ge & \frac{1}{8n^2(n+2)}\int_{\mathbb{R}^n_+}\left(1+x_{n}-2\log(1+x_{n})-\frac{1}{1+x_{n}}\right)(1+x_n)^{-3}|x+e_n|^{-2n-2}x_1^4\ud x\nonumber\\
=&\frac{3\omega_{n-2}}{8n^2(n+2)(n-1)(n+1)}
\int_0^{\infty}\frac{1+x_{n}-2\log(1+x_{n})-\frac{1}{1+x_{n}}}{(1+x_n)^{n+2}}\ud x_n\int_0^\infty\frac{r^{n+2}\ud r}{(1+r^2)^{n+1}}\nonumber\\
=&\frac{3\omega_{n-2}}{8n^2(n+2)(n-1)(n+1)} \frac{2}{n(n+2)(n+1)^2}    \frac{n+1}{4n}B\left(\tfrac{n-1}{2},\tfrac{n+1}{2}\right)\nonumber\\
=& \frac{3\omega_{n-2}}{16n^4(n+2)^2(n+1)^2(n-1)}  B\left(\tfrac{n-1}{2},\tfrac{n+1}{2}\right).\label{eq:quarV2}
\end{align}
Finally,
\begin{align}\label{eq:quarV3}
I_3\ge 0.
\end{align}
Therefore, putting \eqref{eq:C1}, \eqref{eq:linearV}, \eqref{eq:linearV1}, \eqref{eq:linearV2}, \eqref{eq:quarV}, \eqref{eq:quarV1}, \eqref{eq:quarV2} and \eqref{eq:quarV3} together, we obtain
\begin{align*}
&C_1(n)\\
>&\frac{\omega_{n-2}(n-12)B(\frac{n-1}{2},\frac{n+1}{2})}{4n(n-1)(n-2)(n-3)} \\
&+\frac{8n^2(n+2)}{3}\left[\frac{9\omega_{n-2}B(\frac{n-1}{2},\frac{n+1}{2})}{4(n+1)^2n^3(n-1)(n-2)(n-3)}+\frac{3\omega_{n-2}B\left(\frac{n-1}{2},\frac{n+1}{2}\right)}{8(n+2)(n+1)^2n^3(n-1)^2}\right] \\
&+\frac{8n^3(n+2)}{3}\left[\frac{9\omega_{n-2}B(\frac{n-1}{2},\frac{n+1}{2})}{16(n+2)(n+1)^2n^4(n-1)(n-2)}+\frac{3\omega_{n-2}B(\frac{n-1}{2},\frac{n+1}{2})}{16n^4(n+2)^2(n+1)^2(n-1)}\right]\\
=&\omega_{n-2}B\left(\tfrac{n-1}{2},\tfrac{n+1}{2}\right)\frac{n^2-8n-5}{4(n+2)(n+1)(n-1)^2(n-3)}.
\end{align*}
Hence,
\begin{align*}
C_1(n)>0 \quad \mathrm{~~if~~}n \geq 9.
\end{align*}

\begin{remark}
Using Lemma \ref{lem:calculus}, one can check that the function
\[
\frac{x_{n+4}}{4n}|x+e_{n+4}|^{-n},\quad x\in\R^{n+4}_+,
\]
is a  \emph{supersolution} of \eqref{eq:subsolution1}, and thus,
\[
\widetilde V(x)\le \frac{x_{n+4}}{4n}|x+e_{n+4}|^{-n} \quad \mathrm{in~~}\overline {\R_+^{n+4}}.
\]
Therefore,
\begin{align*}
C_1(n)=&\frac{\omega_{n-2}(n-12)B(\frac{n-1}{2},\frac{n+1}{2})}{4n(n-1)(n-2)(n-3)}\\
&+\frac{8n^2(n+2)}{3}\int_{\Rn}x^2_n|x+e_n|^{-(n+4)}V(|x'|,x_n)x_1^4 \ud x \\
&+\frac{8n^3(n+2)}{3}\int_{\Rn}|x+e_n|^{-4}V^2(|x'|,x_n)x_1^4\ud x.\\
\le & \frac{(n-12)\omega_{n-2}B(\frac{n-1}{2},\frac{n+1}{2})}{4n(n-1)(n-2)(n-3)}\\
&+\frac{8n^2(n+2)}{3} \frac{9\omega_{n-2}B(\frac{n-1}{2},\frac{n+1}{2})}{16(n+1)n^3(n-1)(n-2)(n-3)}\\
&+\frac{8n^3(n+2)}{3}\frac{3\omega_{n-2}B(\frac{n-1}{2},\frac{n+1}{2})}{64(n+1)n^4(n-1)(n-2)}\\
=&\frac{(3n^2-11n-6)\omega_{n-2}B(\frac{n-1}{2},\frac{n+1}{2})}{8(n+1)n(n-1)(n-2)(n-3)}<0 \quad\mathrm{if~~}n\le 4.
\end{align*}
Hence, the best possible dimension one can achieve in this case under the present proof is $n\ge 5$.
\end{remark}

\subsection{Umbilic boundary in dimensions $7\le n\le 9$}
We first need to find a sub-solution of
\begin{align*}
\begin{cases}
\displaystyle -\D \widetilde \Lambda=x_{n+4}^2|x+e_{n+4}|^{-n-2}, &\quad \mathrm{in~~} \R^{n+4}_+, \\
\displaystyle  \widetilde \Lambda=0,   &\quad \mathrm{on~~} \pa \R^{n+4}_+.  
\end{cases}
\end{align*}
According to Lemma \ref{lem:calculus}, we will choose $\alpha=n$, and search for the solution of
\begin{equation}\label{eq:ode1um}
\begin{cases}
\alpha (n+2-\alpha) f(s) +2\alpha(1+s)f'(s)=s^2,\\
f(0)=0.
\end{cases}
\end{equation}
The solution of \eqref{eq:ode1um} is 
\[
\frac{s^3}{6n(1+s)},
\]
which is a convex function. So we have
\begin{align*}
-\D \left(\frac{x_{n+4}^3}{6n(1+x_{n+4})}|x+e_{n+4}|^{-n}\right)
   = x_{n+4}^2|x+e_{n+4}|^{-n-2}-\frac{1}{3n}\left(1-\frac{1}{(1+x_{n+4})^3}\right)|x+e_{n+4}|^{-n}.
   \end{align*}
Set
\[
\tilde u_2(x):= \widetilde \Lambda(x) - \frac{x_{n+4}^3}{6n(1+x_{n+4})}|x+e_{n+4}|^{-n},\quad x\in\overline{\R^{n+4}_+}.
\]
Then it becomes
\begin{align}\label{eq:u_2}
\begin{cases}
\displaystyle-\D \tilde u_2
   = \frac{1}{3n}\left(1-\frac{1}{(1+x_{n+4})^3}\right)|x+e_{n+4}|^{-n} &~~ \mathrm{in~~} \mathbb{R}_+^{n+4},\\
\displaystyle \tilde u_2=0, &~~ \mathrm{on~~} \pa \mathbb{R}_+^{n+4}.
\end{cases}
\end{align}
So it follows from the maximum principle that
\begin{align}\label{subsol:ineq:1}
\tilde u_2(x)=\widetilde \Lambda(x) - \frac{x_{n+4}^3}{6n(1+x_{n+4})}|x+e_{n+4}|^{-n}\geq 0 \quad \mathrm{in~~} \overline{\mathbb{R}_+^{n+4}}. 
\end{align}

Now we are going to give a better estimate of  $C_2(n)$ defined in \eqref{eq:C2}, which is
\begin{equation*}
\begin{split}
C_2(n):=&\frac{3(n-10)\omega_{n-2}B(\frac{n-1}{2},\frac{n+1}{2})}{2n(n-1)(n-2)(n-3)(n-4)(n-5)}\\
&+\frac{2n^2(n+2)}{3}\int_{\R^n_+}|x+e_n|^{-n-4}\Lambda(|x'|,x_n)x_n^3x_1^4\ud x \\
&+\frac{2n^3(n+2)}{3}\int_{\R^n_+}|x+e_n|^{-4}x_1^4\Lambda^2(|x'|,x_n)\ud x.
\end{split}
\end{equation*}

Let
\[
u_1(r,s)= \tilde u_1 (r,0,\cdots,0,s),
\]
and recall
\[
\Lambda(r,s)= \widetilde \Lambda(r,0,\cdots,0,s).
\]
We have
\begin{equation}\label{eq:linearL}
\begin{split}
& \int_{\R^n_+}|x+e_n|^{-n-4}\Lambda(|x'|,x_n)x_n^3x_1^4\ud x \\
 =&\frac{1}{6n}\int_{\R^n_+}|x+e_n|^{-2n-4}\frac{x_{n}^6}{1+x_{n}}x_1^4\ud x + \int_{\R^n_+}|x+e_n|^{-n-4}x_n^3x_1^4u_2(|x'|,x_n)\ud x.
\end{split}
\end{equation}
 On one hand, 
\begin{align}
\frac{1}{6n}\int_{\R^n_+}|x+e_n|^{-2n-4}\frac{x_{n}^6}{1+x_{n}}x_1^4\ud x =&\frac{\omega_{n-2}}{2n(n-1)(n+1)}\int_0^\infty\frac{x_n^6\ud x_n}{(1+x_n)^{n+2}}\int_0^\infty\frac{r^{n+2}\ud r}{(1+r^2)^{n+2}}\nonumber \\
=&\frac{45\omega_{n-2}B(\frac{n-1}{2},\frac{n+1}{2})}{(n+1)^2n^3(n-1)(n-2)(n-3)(n-4)(n-5)}.\label{eq:linearL1}
\end{align}
On the other hand,  if we let
$\tilde v_2$ be the solution (decay to zero at infinity) of
\begin{align*}
\begin{cases}
\displaystyle -\D \tilde  v_2=x_{n+4}^3|x+e_{n+4}|^{-n-4},& \mathrm{in~~} \R^{n+4}_+, \\
\displaystyle \tilde  v_2=0, & \mathrm{on~~} \pa \R^{n+4}_+,
\end{cases}
\end{align*}
and notice that $\tilde v_2$ is radial in the $x'=(x_1,\cdots,x_{n+3})$ variable, then it follows from Lemma \ref{lem:changeofvariables} and \eqref{eq:u_2} that
\begin{align*}
\int_{\Rn}x^3_n|x+e_n|^{-n-4}u_2(|x'|,x_n)x_1^4 \ud x=&\frac{1}{3n}\int_{\mathbb{R}^n_+} v_2(|x'|,x_n)\left(1-\frac{1}{(1+x_{n})^3}\right)|x+e_{n}|^{-n}x_1^4 \ud x,
\end{align*}
where 
\[
v_2(r,s)=\tilde v_2(r,0,\cdots,0,s).
\]

We will give a lower bound on $\tilde v_2$. According to Lemma \ref{lem:calculus}, we will choose $\alpha=n+2$,  and search for the solution of
\begin{equation}\label{eq:ode2um}
\begin{cases}
2\alpha(1+s)f'(s)=s^3,\\
f(0)=0.
\end{cases}
\end{equation}
The solution of \eqref{eq:ode2um} is 
\[
\frac{1}{2(n+2)}\left[\frac{(1+s)^3}{3}-\frac{3(1+s)^2}{2}+3(1+s)-\log(1+s)-\frac{11}{6}\right],
\]
which is a convex function. Hence, we have
\begin{align*}
-\Delta \left[\frac{1}{2(n+2)}h_1(x_{n+4})|x+e_{n+4}|^{-n-2}\right] \le x_{n+4}^3|x+e_{n+4}|^{-n-4}\quad\quad \mathrm{in~~} \R^{n+4}_+,
\end{align*}
where
\[
h_1(s)=\frac{(1+s)^3}{3}-\frac{3(1+s)^2}{2}+3(1+s)-\log(1+s)-\frac{11}{6}.
\]
Therefore, we have
\begin{align*}
\tilde v_2(x)\ge \frac{1}{2(n+2)}h_1(x_{n+4})|x+e_{n+4}|^{-n-2} \quad \mathrm{in~~}\overline{\R^{n+4}_+}.
\end{align*}
Hence,
\begin{align*}
&\int_{\Rn}x^3_n|x+e_n|^{-n-4}u_2(|x'|,x_n)x_1^4 \ud x\\
=&\frac{1}{3n}\int_{\mathbb{R}^n_+} v_2(|x'|,x_n)\left(1-\frac{1}{(1+x_{n})^3}\right)|x+e_{n}|^{-n}x_1^4 \ud x\\
\ge & \frac{1}{6n(n+2)} \int_{\mathbb{R}^n_+} h_1(x_{n})\left(1-\frac{1}{(1+x_{n})^3}\right)|x+e_{n}|^{-2n-2}x_1^4\ud x\\
\ge & \frac{1}{6n(n+2)}\frac{3\omega_{n-2}}{(n-1)(n+1)}\int_0^\infty h_1(x_n)\left(1-\frac{1}{(1+x_{n})^3}\right)(1+x_n)^{-n+1}\ud x_n \int_0^\infty\frac{r^{n+2}}{(1+r^2)^{n+1}}\ud r\\
= & \frac{1}{6n(n+2)}\frac{3\omega_{n-2}}{(n-1)(n+1)} \frac{n+1}{4n}B\left(\tfrac{n-1}{2},\tfrac{n+1}{2}\right) \int_0^\infty h_1(x_n)\left(1-\frac{1}{(1+x_{n})^3}\right)(1+x_n)^{-n+1}\ud x_n.
\end{align*}
It is elementary to calculate that
\begin{align*}
\int_0^\infty h_1(x_n)\left(1-\frac{1}{(1+x_{n})^3}\right)(1+x_n)^{-n+1}\ud x_n=\frac{18(5n^3-24n^2+51n-40)}{(n-5)(n-4)(n-3)(n-2)^2(n-1)n(n+1)^2}.
\end{align*}
Hence,
\begin{align}
&\int_{\Rn}x^3_n|x+e_n|^{-n-4}u_2(|x'|,x_n)x_1^4 \ud x\nonumber\\
\ge &\frac{\omega_{n-2}B\left(\tfrac{n-1}{2},\tfrac{n+1}{2}\right)}{4n^2(n+2)(n-1)} \frac{9(5n^3-24n^2+51n-40)}{(n-5)(n-4)(n-3)(n-2)^2(n-1)n(n+1)^2}.\label{eq:linearL2}
\end{align}

Lastly,
\begin{align}
 &\int_{\R^n_+}|x+e_n|^{-4}x_1^4\Lambda^2(|x'|,x_n)\ud x \nonumber\\
=&\int_{\R^n_+}|x+e_n|^{-4}x_1^4\left(\frac{1}{6n}\frac{x_{n}^3}{1+x_{n}}|x+e_{n}|^{-n}+u_2(|x'|,x_n)\right)^2\ud x \nonumber\\
=&\int_{\R^n_+}|x+e_n|^{-4}x_1^4\left(\frac{1}{6n}\frac{x_{n}^3}{1+x_{n}}|x+e_{n}|^{-n}\right)^2\ud x\nonumber \\
&+\frac{1}{3n}\int_{\R^n_+}|x+e_n|^{-n-4}x_1^4 \frac{x_{n}^3}{1+x_{n}}u_2(|x'|,x_n)\ud x + \int_{\R^n_+}|x+e_n|^{-4}x_1^4u_2^2(|x'|,x_n)\ud x \nonumber\\
=&II_1+II_2+II_3.\label{eq:quarL}
\end{align}
For $II_1$, we have
\begin{align}
II_1=&\frac{\omega_{n-2}}{12(n+1)n^2(n-1)}\int_0^\infty\frac{x_n^6 \ud x_n}{(1+x_n)^{n+3}}\int_0^\infty\frac{r^{n+2}\ud r}{(1+r^2)^{n+2}} \nonumber\\
=&\frac{15\omega_{n-2}B(\frac{n-1}{2},\frac{n+1}{2})}{2(n+2)(n+1)^2n^4(n-1)(n-2)(n-3)(n-4)}.\label{eq:quarL1}
\end{align}

For $II_2$, if we let $ \tilde w_1$ be the solution (decay to zero at infinity) of
\begin{align*}
\begin{cases}
\displaystyle -\D \tilde w_2=\frac{x_{n+4}^3}{1+x_{n+4}}|x+e_{n+4}|^{-n-4},& \mathrm{in~~} \R^{n+4}_+, \\
\displaystyle \tilde w_2=0, & \mathrm{on~~} \pa \R^{n+4}_+,
\end{cases}
\end{align*}
and notice that $\tilde w_2$ is radial in the $x'=(x_1,\cdots,x_{n+3})$ variable, then it follows from Lemma \ref{lem:changeofvariables} and \eqref{eq:u_2} that
 \begin{align*}
II_2=\frac{1}{9n^2}\int_{\mathbb{R}^n_+}w_2(|x'|,x_n)\left(1-\frac{1}{(1+x_{n})^3}\right)|x+e_{n}|^{-n}x_1^4\ud x,
\end{align*}
where 
\[
w_2(r,s)= \tilde w_2 (r,0,\cdots,0,s).
\]

We will give a lower bound estimate for $\tilde w_2$.  According to Lemma \ref{lem:calculus}, we will choose $\alpha=n+2$,  and search for the solution of
\begin{equation}\label{eq:ode3um}
\begin{cases}
2\alpha(1+s)f'(s)=\frac{s^3}{1+s},\\
f(0)=0.
\end{cases}
\end{equation}
The solution of \eqref{eq:ode3um} is 
\[
\frac{1}{2(n+2)}\left[\frac{(s+1)^2}{2}-3(s+1)+3\log (s+1)+\frac{1}{s+1}+\frac 32\right],
\]
which is a convex function. Hence, we have
\begin{align*}
- \D\left\{\frac{1}{2(n+2)}h_2(x_{n+4})|x+e_{n+4}|^{-n-2}\right\}  \le \frac{x_{n+4}^3}{1+x_{n+4}}|x+e_{n+4}|^{-n-4}\quad \mathrm{in~~} \R^{n+4}_+,
\end{align*}
where
\[
h_2(s):=\frac{(s+1)^2}{2}-3(s+1)+3\log (s+1)+\frac{1}{s+1}+\frac 32.
\]
By the maximum principle, we have
\[
\tilde w_2(x)\ge \frac{1}{2(n+2)} h_2(x_{n+4})|x+e_{n+4}|^{-n-2} \quad \mathrm{in~~}\overline{\R^{n+4}_+}.
\]
Thus,
\begin{align}
II_2\ge & \frac{1}{18n^2(n+2)}\int_{\R^n_+}|x+e_n|^{-2n-2}x_1^4 h_2(x_{n})\left(1-\frac{1}{(1+x_{n})^3}\right)\ud x\nonumber\\
\ge & \frac{1}{18n^2(n+2)} \frac{3\omega_{n-2}}{(n-1)(n+1)} \frac{n+1}{4n}B\left(\tfrac{n-1}{2},\tfrac{n+1}{2}\right)\nonumber\\
&\cdot  \int_0^\infty h_2(x_n)(1+x_n)^{-n+1}(1-(1+x_n)^{-3})\ud x_n\nonumber\\
=& \frac{3 (5n^3-13n^2+26n-16)\omega_{n-2}B\left(\tfrac{n-1}{2},\tfrac{n+1}{2}\right)}{4(n-4)(n-3)(n-2)^2(n-1)^2n^4(n+1)^2(n+2)^2}.\label{eq:quarL2}
\end{align}
Finally,
\begin{align}\label{eq:quarL3}
II_3\ge 0.
\end{align}
Therefore, putting \eqref{eq:C2}, \eqref{eq:linearL}, \eqref{eq:linearL1}, \eqref{eq:linearL2}, \eqref{eq:quarL}, \eqref{eq:quarL1}, \eqref{eq:quarL2} and \eqref{eq:quarL3} together, we obtain
\begin{align*}
C_2(n)>&\frac{3(n-10)\omega_{n-2}B(\frac{n-1}{2},\frac{n+1}{2})}{2n(n-1)(n-2)(n-3)(n-4)(n-5)}\\
&+\frac{30(n+2)\omega_{n-2}B(\frac{n-1}{2},\frac{n+1}{2})}{(n+1)^2n(n-1)(n-2)(n-3)(n-4)(n-5)}\\
&+ \frac{3(5n^3-24n^2+51n-40)\omega_{n-2}B\left(\tfrac{n-1}{2},\tfrac{n+1}{2}\right)}{2(n-5)(n-4)(n-3)(n-2)^2(n-1)^2n(n+1)^2}\\
&+\frac{10\omega_{n-2}B(\frac{n-1}{2},\frac{n+1}{2})}{2(n+1)^2n(n-1)(n-2)(n-3)(n-4)}\\
&+\frac{(5n^3-13n^2+26n-16)\omega_{n-2}B\left(\tfrac{n-1}{2},\tfrac{n+1}{2}\right)}{2(n-4)(n-3)(n-2)^2(n-1)^2n(n+1)^2(n+2)}\\
=&\frac{(3n^5-33n^4+106n^3-119n^2+59n-40)\omega_{n-2}B(\frac{n-1}{2},\frac{n+1}{2})}{2(n+1)^2n(n-1)^2(n-2)^2(n-3)(n-4)(n-5)}\\
&+\frac{(5n^3-13n^2+26n-16)\omega_{n-2}B\left(\tfrac{n-1}{2},\tfrac{n+1}{2}\right)}{2(n-4)(n-3)(n-2)^2(n-1)^2n(n+1)^2(n+2)}\\
=& \frac{(3n^3-24n^2+27n+34)\omega_{n-2}B\left(\tfrac{n-1}{2},\tfrac{n+1}{2}\right)}{2(n-5)(n-4)(n-3)(n-2)(n-1)^2(n+1)(n+2)}.
\end{align*}
Hence,
\begin{align*}
C_2(n)>0, \quad \mathrm{~~if~~}n \geq 7.
\end{align*}

\begin{remark}
According to \eqref{eq:estimateofleading2}, if $n\le 6$, then the dominant error term in \eqref{JX-volume expansion} will be of order $\e^4|\log\e|$, and thus, the right hand side of \eqref{JX-volume expansion} needs more delicate expansion in $\e$. Therefore, in this umbilic case, $n=7$ is the best one can do under the present proof.
\end{remark}


\begin{thebibliography}{99}


\bibitem{Almaraz1}
S. Almaraz, \textit{An existence theorem of conformal scalar-flat metrics on manifolds with boundary},
Pacific J. Math. 248 (2010)  no. 1, 1-22.


\bibitem{Aubin} T. Aubin,
        \emph{\'Equations différentielles non lin\'eaires et probl\`eme de Yamabe concernant la courbure scalaire}. J. Math. Pures Appl. (9) \textbf{55} (1976), no. 3, 269--296.
        
\bibitem{ChenSophie}
S. Chen, \textit{Conformal deformation to scalar flat metrics with constant mean curvature on the boundary in higher dimensions}, arXiv:0912.1302v2.


\bibitem{Chen-Ruan-Sun}
X. Chen, Y. Ruan and L. Sun, \textit{The Han-Li conjecture in constant scalar curvature and constant boundary mean curvature problem on compact manifolds}, preprint (2018), arXiv:1805.09597.

\bibitem{escobar1}
J. Escobar, \textit{Conformal deformation of a Riemannian metric to a scalar flat metric with constant mean curvature on the boundary}, Ann. of Math. (2) 136 (1992), no. 1, 1-50.

\bibitem{escobar6}
J. Escobar, \textit{Conformal metrics with prescribed mean curvature on the boundary}, Calc. Var. Partial Differential Equations 4 (1996), no. 6, 559-592.


\bibitem{Gluck-Zhu}
M. Gluck and M. Zhu \textit{An extension operator on bounded domains and applications}, preprint (2017), arXiv:1709.03649.

\bibitem{han-li1}
Z. C. Han and Y. Y. Li, \textit{The existence of conformal metrics with constant scalar curvature and constant boundary mean curvature}, Comm. Anal. Geom. 8 (2000), no. 4, 809-869.

\bibitem{han-li2}
 Z. C. Han and Y. Y. Li, \textit{The Yamabe problem on manifolds with boundary: existence and compactness results}, Duke Math. J. 99 (1999), no. 3, 489-542.
 
\bibitem{Hang-Wang-Yan}
 F. Hang, X. Wang and X. Yan, \textit{An integral equation in conformal geometry}, Ann. I. H. Poincar\'{e} Anal. Non lin\'{e}aires 26 (2009), no. 1, 1-21.
 
 \bibitem{Hang-Wang-Yan2}
 F. Hang, X. Wang and X. Yan, \textit{Sharp integral inequalities for harmonic functions}, Comm. Pure Appl. Math. 61 (2008), no. 1, 54-95.


\bibitem{Jin-Xiong}
T. Jin and J. Xiong, \textit{On the isoperimetric quotient over scalar-flat conformal classes}, Comm. Partial Differential Equations 43 (2018), no. 12, 1737-1760.


\bibitem{marques1}
 F.  Marques, \textit{Existence results for the Yamabe problem on manifolds with boundary}, Indiana Univ. Math. J. 54 (2005), no.6, 1599-1620. 
 
 \bibitem{marques3}
 F. Marques, \textit{Conformal deformations to scalar-flat metrics with constant mean curvature on the boundary}, Comm. Anal. Geom. 15 (2007), no. 2, 381-405.

\bibitem{MN} M. Mayer and C.B. Ndiaye, 
         \emph{Barycenter technique and the Riemann mapping problem of Cherrier-Escobar}.
         J. Differential Geom. \textbf{107} (2017), no. 3, 519--560. 


\bibitem{Schoen} R. Schoen,
         \emph{Conformal deformation of a Riemannian metric to constant scalar curvature}. J. Differential Geom. \textbf{20} (1984), no. 2, 479--495.

\bibitem{Trudger} N.S. Trudinger, \emph{Remarks concerning the conformal deformation of Riemannian structures on compact manifolds}. Ann. Scuola Norm. Sup. Pisa (3) \textbf{22} (1968),  265--274.

\bibitem{Xiong}
 J. Xiong, \textit{ On a conformally invariant integral equation involving Poisson kernel}, Acta Math. Sin. (Engl. Ser.) 34 (2018), no. 4, 681-690.

\bibitem{Yamabe} H. Yamabe,
         \emph{On a deformation of Riemannian structures on compact manifolds}. Osaka Math. J. \textbf{12} (1960), 21--37.

\end{thebibliography}
\end{document}